\documentclass[11pt,reqno]{amsart}
\usepackage{mathrsfs}
\usepackage[breaklinks]{hyperref}
\usepackage{amscd}
\usepackage[headheight=110pt,top=1.0in, bottom=1.0in, left=0.8in, right=0.8in]{geometry}

\usepackage{graphicx}
\usepackage{epstopdf}
\usepackage{inputenc}

\usepackage{url}
\usepackage{amssymb}
\usepackage{amsmath}
\allowdisplaybreaks

\renewcommand{\Re}{\operatorname{Re}}

\renewcommand{\(}{\left\(}
\renewcommand{\)}{\right\)}
\renewcommand{\[}{\left\[}
\renewcommand{\]}{\right\]}
\newcommand{\Z}{\mathbb{Z}}
\newcommand{\K}{\mathbb{K}}
\newcommand{\Q}{\mathbb{Q}}

\newcommand{\N}{\mathbb{N}}

\numberwithin{equation}{section}
\theoremstyle{plain}
\newtheorem{theorem}{Theorem}[section]
\newtheorem{lemma}[theorem]{Lemma}

\newtheorem{corollary}[theorem]{Corollary}

\theoremstyle{remark}
\newtheorem*{remark}{{\bf Remark}}

   \makeatletter
\def\proof{\@ifnextchar[{\@oproof}{\@nproof}}
\def\@oproof[#1][#2]{\trivlist\item[\hskip\labelsep\textbf{#2 Proof of\
#1.}~]\ignorespaces}
\def\@nproof{\trivlist\item[\hskip\labelsep\textit{Proof.}~]\ignorespaces}

\makeatother

\makeatletter
\def\@tocline#1#2#3#4#5#6#7{\relax
  \ifnum #1>\c@tocdepth 
  \else
    \par \addpenalty\@secpenalty\addvspace{#2}%
    \begingroup \hyphenpenalty\@M
    \@ifempty{#4}{%
      \@tempdima\csname r@tocindent\number#1\endcsname\relax
    }{%
      \@tempdima#4\relax
    }%
    \parindent\z@ \leftskip#3\relax \advance\leftskip\@tempdima\relax
    \rightskip\@pnumwidth plus4em \parfillskip-\@pnumwidth
    #5\leavevmode\hskip-\@tempdima
      \ifcase #1
       \or\or \hskip 1em \or \hskip 2em \else \hskip 3em \fi%
      #6\nobreak\relax
    \dotfill\hbox to\@pnumwidth{\@tocpagenum{#7}}\par
    \nobreak
    \endgroup
  \fi}
\makeatother

\usepackage{bigints}
\usepackage{suffix}
\usepackage{mathtools}
\DeclarePairedDelimiterX\MeijerM[3]{\lparen}{\rparen}%
{\begin{smallmatrix}#1 \\ #2\end{smallmatrix}\delimsize\vert\,#3}

\newcommand\MeijerG[8][]{%
  G^{\,#2,#3}_{#4,#5}\MeijerM[#1]{#6}{#7}{#8}}

\WithSuffix\newcommand\MeijerG*[7]{%
  G^{\,#1,#2}_{#3,#4}\MeijerM*{#5}{#6}{#7}}

\numberwithin{theorem}{section}
\numberwithin{equation}{section}

\begin{document}
\title[]{A note on odd zeta values over any number field and Extended Eisenstein series}
\author{Soumyarup Banerjee, Rajat Gupta and Rahul Kumar}
\address{Department of Mathematics and Statistics, Indian Institute of Science Education and Research Kolkata, Mohanpur-741246, West Bengal, India.}
\email{soumya.tatan@gmail.com}

\address{Department of Mathematics, University of Texas at Tyler, TX 75799, U.S.A.}\email{rgupta@uttyler.edu}

\address{Department of Mathematics, The Pennsylvania State University, University Park, PA 16802, U.S.A.}
\email{rjk6031@psu.edu}

\begin{abstract}
In this article, we have studied transformation formulas of zeta function at odd integers over an arbitrary number field which in turn generalizes Ramanujan's identity for the Riemann zeta function. The above transformation leads to a new number field extension of Eisenstein series, which satisfies the transformation $z \mapsto -1/z$ like an integral weight modular form over SL$_2(\Z)$. The results provide number of important applications, which are important in studying the behaviour of odd zeta values as well as Lambert series in an arbitrary number field.
\end{abstract}
\maketitle
\vspace{-0.8cm}
\section{Introduction}
One of the central objects in analytic number theory is the  Riemann zeta function $\zeta(s)$, which plays a crucial role in studying the distribution of primes and has applications in various branches of physics, probability theory, and applied statistics. Over the years, the zeta values and its arithmetic behaviour have been studied extensively by many mathematicians. Euler discovered the zeta values at positive even integral arguments in 1734, which precisely states that for any natural number $m$,
\begin{equation}
\zeta(2m) = (-1)^{m+1} \frac{(2\pi)^{2m} B_{2m}}{2(2m)!}, \label{Euler's formula}
\end{equation}
where $B_m$ denotes the $m$-th Bernoulli number. Thus, it follows from the above formula that the zeta values at any positive even integer are transcendental due to the well-known fact that $\pi$ is transcendental and $B_m$ is rational.

The arithmetic nature of $\zeta(s)$ at odd integers is still mysterious. Ap{\'e}ry \cite{Apery1}, \cite{Apery2} established the irrationality of $\zeta(3)$. Recently, Zudilin \cite{Zudilin} has shown that at least one of the four numbers $\zeta(5)$, $\zeta(7)$, $\zeta(9)$, $\zeta(11)$ is irrational. In this direction, a celebrated result was obtained by Lerch \cite{lerch}, which precisely states that, for $m\geq0$,
\begin{align}\label{lerch}
\zeta(4m+3)=2^{4m+2}\pi^{4m+3}\sum_{j=0}^{2m+2}(-1)^{j+1}\frac{B_{2j}B_{4m+4-2j}}{(2j)!(4m+4-2j)!}-2\sum_{n=1}^\infty\frac{n^{-4m-3}}{e^{2\pi n}-1}.
\end{align}
Since the infinite series is rapidly convergent in the above formula, Berndt and Straub \cite{BS} state that ``\emph{$\zeta(4m+3)$ is almost a rational multiple of $\pi^{4m+3}$}".

Ramanujan \cite[p.~173, Ch.~14, Entry~21(i)]{Ramanujan} obtained an elegant generalization of Lerch identity \eqref{lerch}, namely,  for any non-zero integer $m$ and $\alpha, \beta >0$ with $\alpha \beta = \pi^2$, 
\begin{align}\label{Ramanujan formula}
\alpha^{-m}\left\lbrace \frac{1}{2}\zeta(2m+1) + \sum_{n=1}^\infty\frac{n^{-2m-1}}{e^{2n\alpha} - 1} \right\rbrace
&= (-\beta)^{-m}\left\lbrace \frac{1}{2}\zeta(2m+1) + \sum_{n=1}^\infty\frac{n^{-2m-1}}{e^{2n\beta} - 1} \right\rbrace \nonumber\\
&\quad+\frac{1}{\pi^{2m+2}}\sum_{j=0}^{m+1}(-1)^{m+j}\zeta(2j)\zeta(2m-2j+2)\alpha^{m+1-j}\beta^j.
\end{align}
The above identity encodes even integral weight Eisenstein series transformation $z \mapsto -1/z$ over the full modular group as well as the Eichler integrals and has some important applications in theoretical computer science [cf. \cite{KP}]. Ramanujan's identity \eqref{Ramanujan formula} has been generalized in different directions (cf. \cite{brad}, \cite{dg}, \cite{DGKM}, \cite{dkk1}, \cite{dk}, \cite{DM}, \cite{mg} and \cite{komori}). 

The zeta values at even integers over any number field have been studied extensively. Klingen \cite{Klingen} and Siegel \cite{Siegel} generalized the result \eqref{Euler's formula} of Euler for any totally real number field. The result precisely states as for any totally real number field $\K$ of degree $d$ with discriminant $D_\K$,  
\begin{equation}\label{KlingenSiegel}
\zeta_\K(2m) = \frac{q_m\pi^{2md}}{\sqrt{D_K}}\qquad (m \in \N ),
\end{equation} 
where $q_m$ is some fixed non-zero rational number. 
Zagier \cite{Zagier} obtained a closed form expression for $\zeta_\K(2)$ over any number field $\K$. Here, we investigate transformation formula for odd zeta values over  an arbitrary number field, which in turn, generalizes \eqref{Ramanujan formula}. Before stating the new transformation formula, we define the Dedekind zeta function. 

Let $\K$ be any number field with extension degree $[\K : \Q] = d$ and signature $(r_1, r_2)$ (i. e., $d = r_1+2r_2$) and $D_\K$ denotes the absolute value of the discriminant of $\K$. Let $\mathcal{O}_\K$ be its ring of integers and $\mathfrak{N}$ be the norm map of $\K$ over $\Q$. Then the \begin{it}Dedekind zeta function\end{it} attached to number field $\K$ is defined by 
$$
\zeta_\K(s) :=\sum_{\mathsf{a}\subset\mathcal{O}_\K}\frac{1}{\mathfrak{N}(\mathsf{a})^s},
$$
for $\mathrm{Re}(s)>1$, where $\mathsf{a}$  runs over the non-zero integral ideals  of $\mathcal{O}_\K$.
If $V_\K(m)$ counts the number of non-zero integral ideals in $\mathcal{O}_\K$ with norm $m$, then $\zeta_\K(s)$ can also be written as 
\begin{align}\label{dedekind defn}
\zeta_\K(s)=\sum_{m=1}^\infty \frac{V_\K(m)}{m^s},\qquad(\mathrm{Re}(s)>1).
\end{align}
The function $\zeta_\K(s)$ extends analytically to the entire complex plane except for a simple pole at $s = 1$. We denote the residue at $s = 1$ by $H$, which can be expressed by the analytic class number formula,
\begin{equation*}
H = \frac{2^{r_1} (2\pi)^{r_2} h R}{w \sqrt{D_\K}},
\end{equation*}
where $h$ denotes the class number, $R$ is the regulator and $w$ is the number of roots of unity in $\K$. 
Throughout, we will stick with the above notations for an arbitrary number field $\K$.


Ramanujan's identity \eqref{Ramanujan formula} mainly involves the Lambert series
\begin{equation}\label{Lambertseries}
\sum_{n=1}^\infty \frac{n^a}{e^{ny}-1} = \sum_{n=1}^\infty \sigma_a(n) e^{-ny},
\end{equation}
where $\sigma_a(n) := \sum_{d \mid n} d^a$. The above series plays a crucial role in the theory of modular forms as for $a = 2m-1$ with $m\in \N$ and $y = -2 \pi iz$ either of the above series is  essentially an Eisenstein series of weight $2m$ over the full modular group. Here, the case $m=1$ is  Eisenstein series of weight $2$, which is not exactly the modular form but the quasi-modular form. The series for $a = -2m-1$ with $m\in \N$ and $y = -2 \pi iz$ represents the Eichler integral corresponding to the weight $2m$ Eisenstein series.

Here, we consider an infinite series involving Meijer $G$-function [cf. \S \ref{pr}], which generalizes the above Lambert series \eqref{Lambertseries} for an arbitrary number field $\K$ as
\begin{equation}\label{Lamber series numberfield}
\sum_{\mathsf{a}\subset\mathcal{O}_\K} \mathfrak{N}(\mathsf{a})^a \Omega_\K \left(\frac{\mathfrak{N}(\mathsf{a}) y}{D_{\K}}\right)= \sum_{n=1}^\infty V_\K(n)n^{a} \Omega_\K\left(\frac{n y}{D_\K}\right)
\end{equation}
where,
\begin{align}\label{omega definition}
\Omega_\K(x):=\frac{2^{1-r_1-r_2}}{\pi^{1-\frac{r_1}{2}}}\sum_{j=1}^\infty V_\K(j)\MeijerG*{d+1}{0}{0}{2d}{-}{(0)_{r_1+r_2},\left(\frac{1}{2}\right)_{r_2+1};\left(\frac{1}{2}\right)_{r_1+r_2-1},(0)_{r_2}}{\frac{x^2j^2}{4^d}}.
\end{align}
In particular, for $\K= \Q$, $\Omega_\K(x)$ reduces to $\frac{1}{e^x-1}$, which has been derived in \eqref{OmegaQ}.

In the following theorem we obtain a transformation formula of Dedekind zeta function at odd integers over an arbitrary number field $\K$, which provides \eqref{Ramanujan formula} as a special case when $\K = \Q$.

\begin{theorem}\label{gen formula}
Let $\K$ be an arbitrary number field. For any non-zero integer $m$ and  $\mathrm{Re}(\alpha), \mathrm{Re}(\beta)>0$ with $\alpha\beta=\pi^{2d}$, the following identity holds
\begin{align}\label{gen formula eqn}
&\alpha^{-m}\left\{\frac{H}{2}\zeta_\K(2m+1)+\mathcal{C_\K}\sum_{n=1}^\infty\frac{V_\K(n)}{n^{2m+1}}\Omega_\K\left(\frac{2^d n \alpha}{D_\K}\right)\right\}\nonumber\\
&=(-\beta)^{-m}\left\{\frac{H}{2}\zeta_\K(2m+1)+\mathcal{C_\K}\sum_{n=1}^\infty\frac{V_\K(n)}{n^{2m+1}}\Omega_\K\left(\frac{2^d n \beta}{D_\K}\right)\right\}\nonumber\\
&\quad+\frac{1}{\pi^{(2m+1)d+1}}\sum_{j=0}^{m+1}(-1)^{m+j}\zeta_\K(2j)\zeta_\K(2m-2j+2)\alpha^{m+1-j}\beta^j.
\end{align}
where $ \mathcal{C_\K}=\frac{2^{r_1-1}(2\pi)^{r_2}}{\sqrt{D_\K}}$.
\end{theorem} 
The following corollary provides Ramanujan's identity as a special case of the above theorem.
\begin{corollary}\label{cor1.2}
For $\K = \Q$, Ramanujan's identity \eqref{Ramanujan formula} holds.
\end{corollary}
When $\K$ is any real quadratic field, Franke \cite{Franke} has studied an analogue of Ramanujan's identity but could not get the result in a symmetric form. The following corollary follows from Theorem \ref{gen formula}, which provides the transformation formula in symmetric form.
\begin{corollary}\label{real case}
Let $\K$ be any real quadratic field. Let
\begin{align}
\mathfrak{R}(x):=\sum_{j=1}^\infty V_\K(j)\left(K_0(2 \epsilon \sqrt{j x})+K_0(2\overline{\epsilon}\sqrt{j x})\right),\nonumber
\end{align}
where $\epsilon=e^{i\pi/4}$ and $K_\nu(x)$ denotes the Bessel function of the second kind \textup{(}cf. \S \ref{pr}\textup{)}. If $\mathrm{Re}(\alpha)>0$ and $\mathrm{Re}(\beta)>0$ with $\alpha\beta=\pi^4$ and $m$ be any non-zero integer. Then 
\begin{align}\label{real case eqn}
&\alpha^{-m}\left\{\frac{H}{2}\zeta_\K(2m+1)+\frac{2}{\sqrt{D}_\K}\sum_{n=1}^\infty\frac{V_\K(n)}{n^{2m+1}}\mathfrak{R}\left(\frac{4n\alpha}{D_\K}\right)\right\}\nonumber\\
&=(-\beta)^{-m}\left\{\frac{H}{2}\zeta_\K(2m+1)+\frac{2}{\sqrt{D}_\K}\sum_{n=1}^\infty\frac{V_\K(n)}{n^{2m+1}}\mathfrak{R}\left(\frac{4n\beta}{D_\K}\right)\right\}\nonumber\\
&\qquad+\frac{1}{\pi^{4m+3}}\sum_{j=0}^{m+1}(-1)^{m+j}\zeta_\K(2j)\zeta_\K(2m-2j+2)\alpha^{m+1-j}\beta^j.
\end{align}
\end{corollary}
The following corollary provides transformation formula for $\zeta_\K(2m+1)$, when $\K$ is an imaginary quadratic field.
\begin{corollary}\label{imaginary case}
Let $\K$ be any imaginary quadratic field. Let
\begin{align}\label{im quad kernel}
\mathcal{I}(x):=\frac{2i}{\pi}\sum_{j=1}^\infty V_\K(j)\left(K_0(2 \epsilon \sqrt{j x})-K_0(2\overline{\epsilon}\sqrt{j x})\right).
\end{align}
For $m\in \Z\setminus \{0\}$ and $\mathrm{Re}(\alpha), \mathrm{Re}(\beta)>0$ with $\alpha\beta=\pi^4$, we have
\begin{align}
&\alpha^{-m}\left\{\frac{H}{2}\zeta_\K(2m+1)+\frac{\pi}{\sqrt{D}_\K}\sum_{n=1}^\infty\frac{V_\K(n)}{n^{2m+1}}\mathcal{I}\left(\frac{4 n \alpha}{D_\K}\right)\right\}\nonumber\\
&=(-\beta)^{-m}\left\{\frac{H}{2}\zeta_\K(2m+1)+\frac{\pi}{\sqrt{D}_\K}\sum_{n=1}^\infty\frac{V_\K(n)}{n^{2m+1}}\mathcal{I}\left(\frac{4 n \beta}{D_\K}\right)\right\}\nonumber\\
&\quad+\frac{1}{\pi^{4m+3}}\sum_{j=0}^{m+1}(-1)^{m+j}\zeta_\K(2j)\zeta_\K(2m-2j+2)\alpha^{m+1-j}\beta^j.\nonumber
\end{align}
\end{corollary}
We next generalize Lerch identity \eqref{lerch} for an arbitrary number field $\K$.
\begin{theorem}\label{General Lerch}
Let $\K$ be an arbitrary number field. For any integer $m$, the following identity holds :
\begin{align}\label{General Lerch eqn}
\zeta_\K(4m+3) = \frac{1}{\pi H} \sum_{j=0}^{2m+2} (-1)^{j+1} \zeta_\K(2j) \zeta_\K(4m-2j+4)
 - \frac{2^{r_1}(2\pi)^{r_2}}{H \sqrt{D_\K}}\sum_{n=1}^{\infty} \frac{V_{\K}(n)}{n^{4m+3}}\Omega_{\K}\left(\frac{(2\pi)^d n}{D_{\K}} \right).
\end{align}
\end{theorem}
The arithmetic nature of $\zeta_\K(2m+1)$ for $m$ odd, follows from the above theorem. 
\begin{corollary}\label{transcendental result}
For $\K$ totally real, atleast one of $\zeta_\K(4m+3)$ and $\sum\limits_{n=1}^{\infty} \frac{V_{\K}(n)}{n^{4m+3}}\Omega_{\K}\left(\frac{(2\pi)^d n}{D_{\K}} \right)$ is transcendental.
\end{corollary}
\begin{remark}
When $\K$ is any real quadratic field, the asymptotics of $K_0(x)$ as $x\to\infty$ readily implies that the infinite series in the above corollary is rapidly convergent. Thus, we can conclude that $\zeta_\K(4m+3)$ is ``almost" an algebraic multiple of $\pi^{n}$ for some fixed positive integer $n$.
\end{remark}

We next investigate the applications of Theorem \ref{gen formula} in the theory of modular form. Let $f$ be a holomorphic function on an Poincare upper half-plane $\mathfrak{H}$. For any positive integer $k$, $f$ is said to be a modular form of weight $k$ over SL$_2(\Z)$, if it satisfies 
\begin{align}
(i) \, f(z+1) &= f(z)\nonumber\\
(ii) \, f(-1/z) &=(-z)^k f(z).\label{2property}
\end{align}
In the theory of modular forms, the crucial example of modular forms is the Eisenstein series. For $z \in \mathfrak{H}$, the `normalized Eisenstein series' of even integral weight $k$ over SL$_2(\Z)$ can be defined by the following Fourier series expansion :
\begin{equation}
E_k(z) := 1- \frac{2k}{B_k}\sum_{n=1}^\infty \sigma_{k-1}(n) e^{2\pi i nz}=1- \frac{2k}{B_k}\sum_{n=1}^\infty \frac{n^{k-1}}{ e^{-2\pi i nz}-1}. \nonumber
\end{equation}
We here extend the above series definition of Eisenstein series over an arbitrary number field $\K$. For an even positive integer $k$ and $z \in \mathfrak{H}$, we define 
\begin{equation}\label{General Eisenstein}
\mathcal{G}_{k, \K}(z) := \frac{H}{2\mathcal{C}_\K} \zeta_\K(1-k) + \sum_{n=1}^\infty\frac{V_\K(n)}{n^{1-k}}\Omega_\K\left(-\frac{(2\pi)^d i n z}{D_\K}\right).
\end{equation}
If $\K$ is totally real, the above definition can be normalized as
\begin{equation}
\mathcal{E}_{k, \K}(z) := \frac{2\mathcal{C}_\K}{H\zeta_\K(1-k)} \mathcal{G}_{k, \K}(z) =  1 + \frac{2\mathcal{C}_\K}{H\zeta_\K(1-k)} \sum_{n=1}^\infty\frac{V_\K(n)}{n^{1-k}}\Omega_\K\left(-\frac{(2\pi)^d i n z}{D_\K}\right).\nonumber
\end{equation}
In particular, for $\K= \Q$, $\mathcal{E}_{k, \Q}(z) = E_k(z)$ (using \eqref{OmegaQ} below).
For $\K$ not totally real, the first term in \eqref{General Eisenstein} vanishes. Thus we have
\begin{equation*}
\mathcal{G}_{k, \K}(z) = \sum_{n=1}^\infty\frac{V_\K(n)}{n^{1-k}}\Omega_\K\left(-\frac{(2\pi)^d i n z}{D_\K}\right).
\end{equation*}
In the following theorem, we have shown that the extended Eisenstein series satisfies the transformation formula \eqref{2property} of an integral weight modular form over SL$_2(\Z)$.
\begin{theorem}\label{eisen trans}
Let $\K$ be an arbitrary number field. For $k$ be an even integer greater than 2 and $z \in \mathfrak{H}$, we have
\begin{equation}
\mathcal{G}_{k, \K} (-1/z) = z^{k} \, \mathcal{G}_{k, \K}(z).\nonumber
\end{equation}
\end{theorem}
\begin{remark}In general, $\mathcal{G}_{k, \K}(z)$ is not a periodic function. Thus, it is not exactly a modular form over SL$_2(\Z)$.
\end{remark}
The above theorem can be re-stated in the following symmetric form.
\begin{theorem}\label{eisen trans symm}
Let $\K$ be an arbitrary number field. For any non-zero positive integer $m>1$ and  $\mathrm{Re}(\alpha), \mathrm{Re}(\beta)>0$ with $\alpha\beta=\pi^{2d}$, we have
\begin{align}\label{eisen trans symm eqn}
\alpha^m \sum_{n=1}^\infty\frac{V_\K(n)}{n^{1-2m}}\Omega_\K\left(\frac{2^d n \alpha}{D_\K}\right) - (-\beta)^m \sum_{n=1}^\infty\frac{V_\K(n)}{n^{1-2m}}\Omega_\K\left(\frac{2^d n \beta}{D_\K}\right) = - \frac{H}{2\mathcal{C_\K}}\{\alpha^m - (-\beta)^m\}\zeta_{\K}(1-2m).
\end{align}
In particular, for $\K$ not totally real, we have
\begin{align*}
\alpha^m \sum_{n=1}^\infty\frac{V_\K(n)}{n^{1-2m}}\Omega_\K\left(\frac{2^d n \alpha}{D_\K}\right) = (-\beta)^m \sum_{n=1}^\infty\frac{V_\K(n)}{n^{1-2m}}\Omega_\K\left(\frac{2^d n \beta}{D_\K}\right).
\end{align*}
\end{theorem}
As an easy consequence of the above theorem, the following corollary provides an explicit evaluation of generalized Lambert series.
\begin{corollary}\label{series evaluation}
For an arbitrary number field $\K$ and for any odd natural number $m>1$, 
\begin{equation}\label{series evaluation eqn}
\sum_{n=1}^\infty \frac{V_\K(n)}{n^{1-2m}}\Omega_\K\left(\frac{(2\pi)^d n}{D_\K}\right) = - \frac{H}{2\mathcal{C_\K}}\zeta_{\K}(1-2m).
\end{equation}
In particular, If $\K$ is not totally real, 
\begin{equation*}
\sum_{n=1}^\infty \frac{V_\K(n)}{n^{1-2m}}\Omega_\K\left(\frac{(2\pi)^d n}{D_\K}\right) = 0.
\end{equation*}
\end{corollary}
The weight-2 Eisenstein series $E_2(z)$ over SL$_2(\Z)$ is not exactly a modular form since instead of \eqref{2property}, it satisfies
\begin{equation}
E_2(-1/z) = z^2 E_2(z) + \frac{6z}{\pi i}.\nonumber
\end{equation}
Therefore, it sometimes known as quasi-modular form. In the following theorem, we investigate the transformation of $\mathcal{G}_{2, \K}(z) $.
\begin{theorem}\label{quasi}
Let $\K$ be an arbitrary number field. For $z\in \mathfrak{H}$, 
\begin{equation}\label{quasi1}
\mathcal{G}_{2, \K}(-1/z) = z^2 \mathcal{G}_{2, \K}(z) - \frac{z}{\pi i}\zeta_\K^2(0).
\end{equation}
Equivalently, for $\mathrm{Re}(\alpha), \mathrm{Re}(\beta)>0$ with $\alpha\beta=\pi^{2d}$,
\begin{align}\label{quasi2}
\alpha\sum_{n=1}^\infty n V_\K(n) \Omega_\K\left(\frac{2^d n \alpha}{D_\K}\right) + \beta\sum_{n=1}^\infty n V_\K(n) \Omega_\K\left(\frac{2^d n \beta}{D_\K}\right) = -\frac{H\zeta_\K(-1)}{2\mathcal{C}_\K}(\alpha+\beta) - \frac{\zeta_\K^2(0)}{\pi^{1-d}}.
\end{align}
\end{theorem}
\begin{remark}
If $\K$ is not $\Q$ or an imaginary quadratic field, it is well-known that $\zeta_\K(0)$ vanishes, thus  $\mathcal{G}_{2, \K}(z)$ satisfies \eqref{2property}. 
\end{remark}
The final result generalizes the transformation formula satisfied by the logarithm of the Dedekind eta function \cite[Equation (3.10)]{BLS} for an arbitrary number field $\K$.
\begin{theorem}\label{Generalized Dedekind eta transformation}
Let $\K$ be an arbitrary number field. For $\mathrm{Re}(\alpha), \mathrm{Re}(\beta)>0$ with $\alpha\beta=\pi^{2d}$, we have
\begin{align}
\sum_{n=1}^\infty\frac{V_\K(n)}{n} \Omega_\K\left(\frac{2^d n \alpha}{D_{\K}}\right) - \sum_{n=1}^\infty\frac{V_\K(n)}{n}\Omega_\K\left(\frac{2^d n \beta}{D_{\K}}\right) = \frac{H^2}{4\mathcal{C}_{\K}} \log\left( \frac{\alpha}{\beta}\right) + \frac{\zeta_{\K}(0) \zeta_{\K}(2)}{\pi^{d+1}\mathcal{C}_{\K} } (\alpha- \beta).\nonumber
\end{align}
\end{theorem}

\section{Preliminaries}\label{pr}
In this section, we will gather known results from the literature that will be used in the subsequent proofs. One such result is the functional equation satisfied by the Dedekind zeta function $\zeta_\K(s)$, which is given by: \cite[pp. 254-255]{Lang}.
\begin{equation}\label{functional equation}
\zeta_\K(s) = D_\K^{\frac{1}{2}-s}2^{ds-r_2}\pi^{ds-r_1-r_2}\frac{\Gamma(1-s)^{r_1+r_2}}{\Gamma(s)^{r_2}}\sin\left(\frac{\pi s}{2}\right)^{r_1}\zeta_\K(1-s).
\end{equation}
The following lemma gives the convexity bound of the Dedekind zeta function in the critical region.
\begin{lemma}\label{Convexity bound}
Let $\K$ be any number field. Then
\begin{equation*}
\zeta_\K(\sigma+it) \ll \begin{cases}
|t|^{\frac{d}{2}-d\sigma+\epsilon} D_\K^{\frac{1}{2}-\sigma+\epsilon}, & \text{if} \, \ \sigma \leq 0, \\
|t|^{\frac{d(1-\sigma)}{2}+\epsilon} D_\K^{\frac{1-\sigma}{2}+\epsilon}, & \text{if} \, \ 0 \leq \sigma \leq 1,\\
|t|^\epsilon D_\K^\epsilon, & \text{if} \, \ \sigma \geq 1
\end{cases}
\end{equation*}
holds true for any $\epsilon>0$. 
\end{lemma}
\begin{proof}
This follows from a standard argument by applying the Phragmen-Lindel\"of principle and the functional equation \eqref{functional equation} of the Dedekind zeta function. The details may be found in \cite[Chapter 5]{Iwaniec}, for example.
\end{proof}

The well-known duplication and reflection formulas of the gamma function \cite[p.~46, Equations (3.4) and (3.5)]{temme} are as following
\begin{align}
\Gamma(s)\Gamma\left(s+\frac{1}{2}\right)&=\frac{\sqrt{\pi}}{2^{2s-1}}\Gamma(2s)\label{duplication}\\
\Gamma(s)\Gamma(1-s)&=\frac{\pi}{\sin(\pi s)}.\label{reflection}
\end{align}

We now wrote down the definitions of the Bessel functions. The Bessel functions of the first kind and the second kind of order $\nu$ are defined by \cite[p.~40, 64]{watson-1944a}
\begin{align}
	J_{\nu}(z)&:=\sum_{m=0}^{\infty}\frac{(-1)^m(z/2)^{2m+\nu}}{m!\Gamma(m+1+\nu)} \hspace{9mm} (z,\nu\in\mathbb{C}),\nonumber\\
	Y_{\nu}(z)&:=\frac{J_{\nu}(z)\cos(\pi \nu)-J_{-\nu}(z)}{\sin{\pi \nu}}\hspace{5mm}(z\in\mathbb{C}, \nu\notin\mathbb{Z}),\nonumber
	\end{align}
	along with $Y_n(z)=\lim_{\nu\to n}Y_\nu(z)$ for $n\in\mathbb{Z}$. 
The modified Bessel functions of the first and second kinds are defined by \cite[p.~77, 78]{watson-1944a}
\begin{align}
I_{\nu}(z)&:=
\begin{cases}
e^{-\frac{1}{2}\pi\nu i}J_{\nu}(e^{\frac{1}{2}\pi i}z), & \text{if $-\pi<\arg(z)\leq\frac{\pi}{2}$,}\nonumber\\
e^{\frac{3}{2}\pi\nu i}J_{\nu}(e^{-\frac{3}{2}\pi i}z), & \text{if $\frac{\pi}{2}<\arg(z)\leq \pi$,}
\end{cases}\nonumber\\
K_{\nu}(z)&:=\frac{\pi}{2}\frac{I_{-\nu}(z)-I_{\nu}(z)}{\sin\nu\pi}\label{kbesse}
\end{align}
respectively. When $\nu\in\mathbb{Z}$, $K_{\nu}(z)$ is interpreted as a limit of the right-hand side of \eqref{kbesse}. 

We next define the Meijer $G$-function. The G-function was introduced initially by Meijer as a very general function using a series. Later, it was defined more generally via a line integral in the complex plane (cf. \cite{erd1}) given by 
\begin{equation}\label{G-function}
\begin{aligned}
G^{m, \ n}_{p, \ q}\bigg(\begin{matrix}
a_1, \ldots, a_p \\
b_1, \ldots, b_q
\end{matrix} \ \bigg|\ z\bigg)=\frac{1}{2\pi i}\underset{{(C)}}{\bigints} \frac{\prod\limits_{j=1}^m\Gamma(b_j-s)\prod\limits_{j=1}^n\Gamma(1-a_j+s)}{\prod\limits_{j=m+1}^q\Gamma(1-b_j+s)\prod\limits_{j=n+1}^p\Gamma(a_j-s)}z^s \rm{d}s ,
\end{aligned}
\end{equation}
where $z \neq 0$ and $m$, $n$, $p$, $q$ are integers which satisfy $0 \leq m \leq q$ and $0 \leq n \leq p$. The poles of the integrand must be all simple. Here $(C)$ in the integral denotes the vertical line from $C-i\infty$ to $C+i\infty$  such that all poles of $\Gamma(b_j-s)$ for $ j=1, \ldots, m$, must lie on one side of the vertical line while all poles of $\Gamma(1-a_j+s)$ for $ j=1, \ldots, n$ must lie on the other side. The integral then converges for $|\arg z| < \delta \pi$ where
$$\delta = m+n - \frac{1}{2}(p+q).$$
The integral additionally converges for $|\arg z|= \delta \pi$ if $(q-p)(\Re(s) + 1/2) > \Re (v) + 1$, where 
$$
v = \sum_{j=1}^{q}b_j - \sum_{j=1}^{p}a_j.
$$
Special cases of the $G$-function include many other special functions. 

\section{Transformation formulas for odd zeta values over an arbitrary number field}
In this section, we will prove all the results presented in the introduction. We commence with a lemma that will be utilized as a crucial step in establishing our general Theorem \ref{gen formula}.
\begin{lemma}\label{inverse mellin transform of omega}
For $c=\mathrm{Re}(s)>1$, we have
\begin{align}
\Omega_\K(x)=\frac{1}{2\pi i}\int_{(c)}\frac{\zeta_\K(s)\Gamma(s)^{r_1+r_2}}{\Gamma(1-s)^{r_2}}\cos\left(\frac{\pi s}{2}\right)^{r_1-1}x^{-s}ds.\nonumber
\end{align}
Here, and throughout, $\int_{(c)}ds$ denotes the line integral $\int_{c-i\infty}^{c+i\infty}ds$ with $c=\mathrm{Re}(s)$.
\end{lemma}
\begin{proof}
The definition of Meijer G-function, given in \eqref{G-function}, implies that
\begin{align}\label{mg}
\MeijerG*{d+1}{0}{0}{2d}{-}{(0)_{r_1+r_2},\left(\frac{1}{2}\right)_{r_2+1};\left(\frac{1}{2}\right)_{r_1+r_2-1},(0)_{r_2}}{\frac{x^2j^2}{4^d}}&=\frac{1}{2\pi i}\int_{(L)}\frac{\Gamma(-s)^{r_1+r_2}\Gamma\left(\frac{1}{2}-s\right)^{r_2+1}}{\Gamma\left(\frac{1}{2}+s\right)^{r_1+r_2-1}\Gamma(1+s)^{r_2}}\left(\frac{x^2j^2}{4^d}\right)^s ds,
\end{align}
where $L=\mathrm{Re}(s)<-\frac{1}{2}$, We next replace $s$ by $-s/2$ in \eqref{mg} so as to obtain for $c=\mathrm{Re}(s)>1$, 
\begin{align}\label{refle}
\MeijerG*{d+1}{0}{0}{2d}{-}{(0)_{r_1+r_2},\left(\frac{1}{2}\right)_{r_2+1};\left(\frac{1}{2}\right)_{r_1+r_2-1},(0)_{r_2}}{\frac{x^2j^2}{4^d}}&=\frac{1}{4\pi i}\int_{(c)}\frac{\Gamma\left(\frac{s}{2}\right)^{r_1+r_2}\Gamma\left(\frac{1}{2}+\frac{s}{2}\right)^{r_2+1}}{\Gamma\left(\frac{1}{2}-\frac{s}{2}\right)^{r_1+r_2-1}\Gamma\left(1-\frac{s}{2}\right)^{r_2}}\left(\frac{xj}{2^d}\right)^{-s} ds\nonumber\\
&=\frac{1}{4\pi i}\int_{(c)}\frac{\left(\Gamma\left(\frac{s}{2}\right)\Gamma\left(\frac{1}{2}+\frac{s}{2}\right)\right)^{r_1+r_2}\Gamma\left(\frac{1}{2}+\frac{s}{2}\right)^{1-r_1}}{\left(\Gamma\left(\frac{1}{2}-\frac{s}{2}\right)\Gamma\left(1-\frac{s}{2}\right)\right)^{r_2}\Gamma\left(\frac{1}{2}-\frac{s}{2}\right)^{r_1-1}}\left(\frac{xj}{2^d}\right)^{-s} ds.
\end{align}
Invoking the duplication and reflection formulas \eqref{duplication} and \eqref{reflection} in \eqref{refle} and simplifying, we deduce that
\begin{align}\label{G mellin}
\MeijerG*{d+1}{0}{0}{2d}{-}{(0)_{r_1+r_2},\left(\frac{1}{2}\right)_{r_2+1};\left(\frac{1}{2}\right)_{r_1+r_2-1},(0)_{r_2}}{\frac{x^2j^2}{4^d}}&=\frac{2^{r_1+r_2-1}}{\pi^{\frac{r_1}{2}-1}}\frac{1}{2\pi i}\int_{(c)}\frac{\Gamma(s)^{r_1+r_2}}{\Gamma(1-s)^{r_2}}\cos\left(\frac{\pi s}{2}\right)^{r_1-1}(xj)^{-s}ds.
\end{align}
We substitute the above expression in \eqref{omega definition} to see that
\begin{align}
\Omega_{\mathbb{K}}(x)=\sum_{j=1}^\infty V_\K(j)\frac{1}{2\pi i}\int_{(c)}\frac{\Gamma(s)^{r_1+r_2}}{\Gamma(1-s)^{r_2}}\cos\left(\frac{\pi s}{2}\right)^{r_1-1}(xj)^{-s}ds.\nonumber
\end{align}
Now, by interchanging the order of summation and integration in the above equation, which is justified due to \cite[p.~30, Theorem 2.1]{temme}, and subsequently using \eqref{dedekind defn}, we arrive at the lemma.
\end{proof}

We now provide a proof of our general transformation formula over arbitrary number field.
\begin{proof}[Theorem \textup{\ref{gen formula}}][]
To begin, we establish the convergence of the series on both sides of \eqref{gen formula eqn}. Notably, it suffices to demonstrate the convergence of the series in \eqref{Lamber series numberfield}. From \eqref{G mellin}, for $c>1$ and large $j$, we have
\begin{align}\label{G bound}
\left|\MeijerG*{d+1}{0}{0}{2d}{-}{(0)_{r_1+r_2},\left(\frac{1}{2}\right)_{r_2+1};\left(\frac{1}{2}\right)_{r_1+r_2-1},(0)_{r_2}}{\frac{x^2j^2}{4^d}}\right|&<\frac{2^{r_1+r_2-1}}{\pi^{\frac{r_1}{2}-1}(xj)^c}\int_{-\infty}^\infty\left|\frac{\Gamma(c+it)^{r_1+r_2}}{\Gamma(1-c-it)^{r_2}}\cos\left(\frac{\pi (c+it)}{2}\right)^{r_1-1}dt\right|\nonumber\\
&\ll j^{-c},
\end{align}
which follows upon observing that the integral on the right-hand side of the initial step converges, a fact that can be verified using the Stirling formula \cite[p.~224]{cop}, namely, for $p\leq\sigma\leq q$,
\begin{equation}\label{strivert}
  |\Gamma(s)|=\sqrt{2\pi}|t|^{\sigma-\frac{1}{2}}e^{-\frac{1}{2}\pi |t|}\left(1+O\left(\frac{1}{|t|}\right)\right)
\end{equation}
as $|t|\to \infty$. Utilizing the bound given in \eqref{G bound} in conjunction with \cite[p.~430, Corollary 7.119]{bordell}, we establish the well-definedness of $\Omega_\K(x)$ defined in \eqref{omega definition}. Furthermore, on account of the formula \eqref{strivert} and Lemma \ref{inverse mellin transform of omega}, for $c>\max\{1,1+\mathrm{Re}(a)\}$, we can deduce that, 
\begin{align*}
\Omega_\K\left(ny\right)\ll_y n^{-c},
\end{align*}
as $n\to\infty$. Thus, we obtain
\begin{align*}
\sum_{n=1}^\infty V_\K(n)n^a\Omega_\K\left(ny\right)\ll\sum_{n=1}^\infty\frac{V_\K(n)}{n^{c-a}}.
\end{align*}
As the series on the right-hand side converges, we have established our claim.

We now prove the transformation formula \eqref{gen formula eqn}. Invoking Lemma \ref{inverse mellin transform of omega} in the first step below, we have
\begin{align}\label{L1}
\sum_{n=1}^\infty\frac{V_\K(n)\Omega_\K(ny)}{n^{2m+1}}&=\sum_{n=1}^\infty\frac{V_\K(n)}{n^{2m+1}}\frac{1}{2\pi i}\int_{(c)}\zeta_\K(s)\frac{\Gamma(s)^{r_1+r_2}}{\Gamma(1-s)^{r_2}}\cos\left(\frac{\pi s}{2}\right)^{r_1-1}(ny)^{-s}ds\nonumber\\
&=\frac{1}{2\pi i}\int_{(c)}\zeta_\K(2m+1+s)\zeta_\K(s)\frac{\Gamma(s)^{r_1+r_2}}{\Gamma(1-s)^{r_2}}\cos\left(\frac{\pi s}{2}\right)^{r_1-1}y^{-s}ds,
\end{align} 
where in the last step we first interchanged the order of summation and integration which can be justified by using \cite[p.~30, Theorem 2.1]{temme}, and then used \eqref{dedekind defn}. The functional equation \eqref{functional equation} can be written in the form 
\begin{align}\label{equv fe1}
\frac{\Gamma(s)^{r_1+r_2}}{\Gamma(1-s)^{r_2}}\cos\left(\frac{\pi s}{2}\right)^{r_1-1}\zeta_\K(s)=\frac{D_\K^{\frac{1}{2}-s}2^{r_2-d(1-s)}\pi^{r_1+r_2-d(1-s)}}{\cos\left(\frac{\pi s}{2}\right)}\zeta_\K(1-s),
\end{align}
which transforms \eqref{L1} into
\begin{align}\label{L(y) identity}
\sum_{n=1}^\infty\frac{V_\K(n)\Omega_\K(ny)}{n^{2m+1}}=\sqrt{D_\K}2^{r_2-d}\pi^{r_1+r_2-d}\frac{1}{2\pi i}\int_{(c)}\frac{\zeta_\K(2m+1+s)\zeta_\K(1-s)}{\cos\left(\frac{\pi s}{2}\right)}\left(\frac{yD_\K}{(2\pi)^d}\right)^{-s}ds.
\end{align}
We next shift the line of integration from $c>1$ to $-2m-3<\mu=\mathrm{Re}(s)<-2m-1$. To that end, let us consider the positively oriented contour formed by the line segments $[\mu-iT,c-iT],\ [c-iT,c+iT],\ [c+iT,\mu+iT],\ [\mu+iT,\mu-iT]$. Note that the integrand has simple poles at $s=0, -2m$ and $s=1-2k$, for $0\leq k\leq m+1$. By Cauchy's residue theorem, we have 
\begin{align}\label{cauchy}
&\frac{1}{2\pi i}\left(\int_{c-iT}^{c+iT}+\int_{c+iT}^{\mu+iT}+\int_{\mu+iT}^{\mu-iT}+\int_{\mu-iT}^{c-iT}\right)\frac{\zeta_\K(2m+1+s)\zeta_\K(1-s)}{\cos\left(\frac{\pi s}{2}\right)}\left(\frac{yD_\K}{(2\pi)^d}\right)^{-s}ds\nonumber\\
&=R_0+R_{-2m}+\sum_{k=0}^{m+1}R_{1-2k}.
\end{align}
It is easy to see that the integrals along the horizontal segments vanish as $T\to\infty$ with the help of Lemma \ref{Convexity bound} and the Stirling's formula \eqref{strivert} since the cosine function reduces to the gamma factors using \eqref{reflection}.

Therefore, using the above fact, \eqref{L(y) identity} and \eqref{cauchy}, we arrive at
\begin{align}\label{gen cauchy}
\sum_{n=1}^\infty\frac{V_\K(n)\Omega_\K(ny)}{n^{2m+1}}=\sqrt{D_\K}2^{r_2-d}\pi^{r_1+r_2-d}\left(R_0+R_{-2m}+\sum_{k=0}^{m+1}R_{1-2k}+I(y)\right),
\end{align}
where 
\begin{align}\label{defnI1}
I(y):=\frac{1}{2\pi i}\int_{(\mu)}\frac{\zeta_\K(2m+1+s)\zeta_\K(1-s)}{\cos\left(\frac{\pi s}{2}\right)}\left(\frac{yD_\K}{(2\pi)^d}\right)^{-s}ds.
\end{align}
The residues involved in \eqref{gen cauchy} can be easily evaluated to
\begin{align}\label{residues gen}
R_0&=-H\zeta_\K(2m+1)\nonumber\\
R_{-2m}&=(-1)^{m}H\zeta_\K(2m+1)\left(\frac{yD_\K}{(2\pi)^d}\right)^{2m}\nonumber\\
R_{1-2k}&=(-1)^{k+1}\frac{2}{\pi}\zeta_\K(2m-2k+2)\zeta_\K(2k)\left(\frac{yD_\K}{(2\pi)^d}\right)^{2k-1}.
\end{align}
Employing the change of variable $s=-2m-w$ in \eqref{defnI1}, we have, for $1<c=\mathrm{Re}(s)<3$,
\begin{align}\label{I(y)1}
I(y)&=(-1)^m\left(\frac{yD_\K}{(2\pi)^d}\right)^{2m}\frac{1}{2\pi i}\int_{(c)}\frac{\zeta_\K(1-w)\zeta_\K(2m+1+w)}{\cos\left(\frac{\pi w}{2}\right)}\left(\frac{yD_\K}{(2\pi)^d}\right)^{w}ds\nonumber\\
&=(-1)^m\left(\frac{yD_\K}{(2\pi)^d}\right)^{2m}\left(\sqrt{D_\K}2^{r_2-d}\pi^{r_1+r_2-d}\right)^{-1}\sum_{n=1}^\infty\frac{V_\K(n)}{n^{2m+1}}\Omega_\K\left(\frac{n(2\pi)^{2d}}{yD_\K^2}\right),
\end{align}
where we used \eqref{L(y) identity}. Now equations \eqref{gen cauchy} \eqref{residues gen} and \eqref{I(y)1} together yield
\begin{align}
\sum_{n=1}^\infty\frac{V_\K(n)\Omega_\K(ny)}{n^{2m+1}}&=\sqrt{D_\K}2^{r_2-d}\pi^{r_1+r_2-d}\left\{-H\zeta_\K(2m+1)+(-1)^mH\zeta_\K(2m+1)\left(\frac{yD_\K}{(2\pi)^d}\right)^{2m}\right.\nonumber\\
&\quad\left.+\frac{2}{\pi}\sum_{k=0}^{m+1}(-1)^{k+1}\zeta_\K(2m-2k+2)\zeta_\K(2k)\left(\frac{yD_\K}{(2\pi)^d}\right)^{2k-1}\right\}\nonumber\\
&\quad+(-1)^m\left(\frac{yD_\K}{(2\pi)^d}\right)^{2m}\sum_{n=1}^\infty\frac{V_\K(n)}{n^{2m+1}}\Omega_\K\left(\frac{n(2\pi)^{2d}}{yD_\K^2}\right).\nonumber
\end{align}
First let $y=\frac{2^d\alpha}{D_\K}$ with $\alpha\beta=\pi^{2d}$ in the above equation and then multiply both sides by $\alpha^{-m}$ and simplify to see that
\begin{align}\label{almost there}
&\alpha^{-m}\left\{\sqrt{D_\K}2^{r_2-d}\pi^{r_1+r_2-d}H\zeta_\K(2m+1)+\sum_{n=1}^\infty\frac{V_\K(n)}{n^{2m+1}}\Omega_\K\left(\frac{2^d n\alpha}{D_\K}\right)\right\}\nonumber\\
&=(-\beta)^{-m}\left\{\sqrt{D_\K}2^{r_2-d}\pi^{r_1+r_2-d}H\zeta_\K(2m+1)+\sum_{n=1}^\infty\frac{V_\K(n)}{n^{2m+1}}\Omega_\K\left(\frac{2^d n\beta}{D_\K}\right)\right\}\nonumber\\
&\qquad+\sqrt{D_\K}2^{r_2-d+1}\pi^{r_1+r_2-d-1}\alpha^{-m}\sum_{k=0}^{m+1}(-1)^{k+1}\zeta_\K(2m-2k+2)\zeta_\K(2k)\left(\frac{\alpha}{\pi^d}\right)^{2k-1}.
\end{align}
Finally \eqref{gen formula eqn} follows upon replacing $k$ by $m+1-k$ in the finite sum of \eqref{almost there} and multiplying both sides by $\mathcal{C_{\K}}$.
\end{proof}
We next prove \eqref{Ramanujan formula} from Theorem \ref{gen formula}.
\begin{proof}[Corollary \textup{\ref{cor1.2}}][]
Applying \cite[p.231, Equation (8)]{Luke}, the following Meijer $G$-function reduces to
\begin{align*}
\MeijerG*{2}{0}{0}{2}{-}{0,\frac{1}{2}}{\frac{x^2j^2}{4}}&= \sqrt{2xj} K_{1/2}(xj)\\
&= \sqrt{\pi}e^{-jx}.
\end{align*}
For $\K= \Q$, the degree $d$ is $1$ with $r_1=1$ and $r_2=0$. Thus definition \eqref{omega definition} yields 
\begin{equation}\label{OmegaQ}
\Omega_\K(x) = \frac{1}{e^x-1},
\end{equation}
which concludes our corollary.
\end{proof}

We next present the proof of the transformation formulas of $\zeta_\K(2m+1)$ for quadratic fields.
\begin{proof}[Corollary \textup{\ref{real case}}][]
To prove \eqref{real case eqn}, it is enough to see that $\Omega_\K(jx)=\mathcal{R}(jx)$ when we take $r_1=2$ and $r_2=0$ in Theorem \ref{gen formula}. To show that, note that, for $c>0$,
\begin{align}\label{cov}
\MeijerG*{3}{0}{0}{4}{-}{0,0,\frac{1}{2};\frac{1}{2}}{\frac{x^2j^2}{16}}&=\frac{1}{2\pi i}\int_{(c)}\frac{\Gamma\left(s\right)^2\Gamma\left(\frac{1}{2}+s\right)}{\Gamma\left(\frac{1}{2}-s\right)}\left(\frac{x^2j^2}{16}\right)^{-s}ds\nonumber\\
&=\frac{1}{2\pi i}\int_{(c)}\frac{(\Gamma\left(s\right)\Gamma\left(\frac{1}{2}+s\right))^2}{\Gamma\left(\frac{1}{2}-s\right)\Gamma\left(\frac{1}{2}+s\right)}\left(\frac{x^2j^2}{16}\right)^{-s}ds\nonumber\\
&=\frac{4}{2\pi i}\int_{(c)}\Gamma^2\left(2s\right)\cos\left(\pi s\right)\left(xj\right)^{-2s}ds,
\end{align}
which follows by applying \eqref{duplication} and \eqref{reflection}. Employing the change of variable $s\to s/2$ in \eqref{cov}, we obtain
\begin{align*}
\MeijerG*{3}{0}{0}{4}{-}{0,0,\frac{1}{2};\frac{1}{2}}{\frac{x^2j^2}{16}}&=\frac{2}{2\pi i}\int_{(2c)}\Gamma^2\left(s\right)\cos\left(\frac{\pi s}{2}\right)\left(xj\right)^{-s}ds\nonumber\\
&=\frac{1}{2}\left\{\frac{1}{2\pi i}\int_{(2c)}\Gamma^2\left(s\right)\left(e^{\frac{\pi i}{2}}jx\right)^{-s}ds+\frac{1}{2\pi i}\int_{(c)}\Gamma^2\left(s\right)\left(e^{-\frac{\pi i}{2}}jx\right)^{-s}ds\right\}\nonumber\\
&=2\left\{K_0\left(2\epsilon\sqrt{jx}\right)+K_0\left(2\overline{\epsilon}\sqrt{jx}\right)\right\},
\end{align*}
where we used \cite[p.~115, Formula (11.1)]{ober}
\begin{align}\label{mb}
K_z(x)=\frac{1}{2\pi i}\int_{(c_1)}\Gamma\left(\frac{s-z}{2}\right)\Gamma\left(\frac{s+z}{2}\right)2^{s-2}x^{-s}ds, \quad c_1>|\mathrm{Re}(z)|.
\end{align}
This proves our claim $\Omega_\K(jx)=\mathcal{R}(jx)$, which concludes our result.
\end{proof}

\begin{proof}[Corollary \textup{\ref{imaginary case}}][]
Let $r_1=0$ and $r_2=1$ in Theorem \ref{gen formula}. First we show that $\Omega_\K(x)$ reduces to $\mathcal{I}(x)$. 
Note that 
\begin{align}
\MeijerG*{3}{0}{0}{4}{-}{0,\frac{1}{2},\frac{1}{2};0}{\frac{x^2j^2}{16}}&=\frac{1}{2\pi i}\int_{(c)}\frac{\Gamma\left(s\right)\Gamma^2\left(s+\frac{1}{2}\right)}{\Gamma\left(1-s\right)}\left(\frac{x^2j^2}{16}\right)^{-s}ds\nonumber\\
&=\frac{1}{2\pi i}\int_{(c)}\frac{(\Gamma\left(s\right)\Gamma\left(s+\frac{1}{2}\right))^2}{\Gamma(s)\Gamma\left(1-s\right)}\left(\frac{x^2j^2}{16}\right)^{-s}ds\nonumber\\
&=4\frac{1}{2\pi i}\int_{(c)}\Gamma^2\left(2s\right)\sin(\pi s)\left(xj\right)^{-2s}ds,\nonumber
\end{align}
where in the last step we used  \eqref{duplication} and \eqref{reflection}. The change of variable $s\to s/2$ in the above expression yields
\begin{align}
\MeijerG*{3}{0}{0}{4}{-}{0,\frac{1}{2},\frac{1}{2};0}{\frac{x^2j^2}{16}}
&=2\frac{1}{2\pi i}\int_{(2c)}\Gamma^2\left(s\right)\sin(\pi s/2)\left(xj\right)^{-s}ds\nonumber\\
&=\frac{1}{i}\frac{1}{2\pi i}\int_{(2c)}\Gamma^2\left(s\right)\left(e^{\frac{\pi is}{2}}-e^{-\frac{\pi is}{2}}\right)\left(xj\right)^{-s}ds\nonumber\\
&=i\frac{1}{2\pi i}\left\{\int_{(2c)}\Gamma^2\left(s\right)\left(e^{\frac{\pi i}{2}}xj\right)^{-s}ds-\int_{(2c)}\Gamma^2\left(s\right)\left(e^{-\frac{\pi i}{2}}xj\right)^{-s}ds\right\}\nonumber\\
&=2i\left\{K_0\left(2\epsilon\sqrt{jx}\right)-K_0\left(2\overline{\epsilon}\sqrt{jx}\right)\right\},\nonumber
\end{align}
where we invoked \eqref{mb}. Combining this with \eqref{im quad kernel} and \eqref{omega definition} gives $\Omega_\K(jx)=\mathcal{I}(jx)$. This completes the proof of our result.
\end{proof}

Next, we establish Lerch's result and its consequential Corollary \ref{transcendental result}.
\begin{proof}[Theorem \textup{\ref{General Lerch}}][]
Upon setting $\alpha=\beta=\pi^d$, substituting $m$ with $2m+1$, and simplifying the expression, we arrive at \eqref{General Lerch eqn}.
\end{proof}

\begin{proof}[Corollary \textup{\ref{transcendental result}}][]
For a totally real field, the result of Klingen \cite{Klingen} and Siegel \cite{Siegel} (given in \eqref{KlingenSiegel} above) leads to the conclusion that the finite sum on the right-hand side of Equation \eqref{General Lerch eqn} is transcendental. As a result, this implies that either $\zeta_k(4m+3)$ or $\sum\limits_{n=1}^{\infty} \frac{V_{\K}(n)}{n^{4m+3}}\Omega_{\K}\left(\frac{(2\pi)^d n}{D_{\K}} \right)$ must also be transcendental.
\end{proof}

\section{Extended Eisenstein series over an arbitrary number field}

We first establish the symmetric form of the transformation formula for Eisenstein series $\mathcal{G}_{k,\K}(z)$, as stated in Theorem \ref{eisen trans symm}. Subsequently, we present the proof of Theorem \ref{eisen trans}.
\begin{proof}[Theorem \textup{\ref{eisen trans symm}}][]
By substituting $m$ with $-m$ in \eqref{gen formula eqn} and considering the case where $m\in\mathbb{N}$, we observe that the finite sum becomes empty and hence evaluates to zero. Further rearrangement of the terms leads to the proof of the desired result.

The second part of the proof follows as a straightforward consequence of \eqref{eisen trans symm eqn}, since for a non-totally real field, it is known that $\zeta_\K(1-2m)=0$.
\end{proof}

\begin{proof}[Theorem \textup{\ref{eisen trans}}][]
Note that \eqref{eisen trans symm eqn} can be re-written as
\begin{align*}
\alpha^m \left\{\sum_{n=1}^\infty\frac{V_\K(n)}{n^{1-2m}}\Omega_\K\left(\frac{2^d n \alpha}{D_\K}\right)+\frac{H}{2\mathcal{C}_\K}\zeta_\K(1-2m)\right\} = (-\beta)^m \left\{\sum_{n=1}^\infty\frac{V_\K(n)}{n^{1-2m}}\Omega_\K\left(\frac{2^d n \beta}{D_\K}\right)+\frac{H}{2\mathcal{C}_\K}\zeta_\K(1-2m)\right\}.
\end{align*}
Letting $\alpha=-\pi^diz$ and $\beta=\pi^di/z$ in the above equation and then simplifying, we are led to
\begin{align*}
\sum_{n=1}^\infty\frac{V_\K(n)}{n^{1-2m}}\Omega_\K\left(\frac{(2\pi)^di nz}{D_\K}\right)+\frac{H}{2\mathcal{C}_\K}\zeta_\K(1-2m)= z^{-2m} \left\{\sum_{n=1}^\infty\frac{V_\K(n)}{n^{1-2m}}\Omega_\K\left(\frac{(2\pi)^di n}{zD_\K}\right)+\frac{H}{2\mathcal{C}_\K}\zeta_\K(1-2m)\right\}.
\end{align*}
By utilizing the definition of the Eisenstein series $\mathcal{G}_{k,\K}(z)$ as given in \eqref{General Eisenstein}, we now conclude the proof of the theorem.
\end{proof}

\begin{proof}[Theorem \textup{\ref{series evaluation}}][]
After setting $\alpha=\pi^d=\beta$ and making the substitution $m\rightarrow 2m+1$, we can simplify the expression, which leads to the closed-form evaluation as expressed in \eqref{series evaluation eqn}. 
\end{proof}

We now supply a proof of the result for quasi-modular form.
\begin{proof}[Theorem \textup{\ref{quasi}}][]
The result in \eqref{quasi2} can be readily derived from \eqref{gen formula eqn} by setting $m=1$ and rearranging the terms accordingly.

For a proof of \eqref{quasi1}, we rewrite \eqref{quasi2} as
\begin{align*}
\alpha\left\{\sum_{n=1}^\infty n V_\K(n) \Omega_\K\left(\frac{2^d n \alpha}{D_\K}\right)+\frac{H\zeta_\K(-1)}{2\mathcal{C}_\K}\right\}= -\beta\left\{\sum_{n=1}^\infty n V_\K(n) \Omega_\K\left(\frac{2^d n \beta}{D_\K}\right)+\frac{H\zeta_\K(-1)}{2\mathcal{C}_\K}\right\} - \frac{\zeta_\K^2(0)}{\pi^{1-d}}.
\end{align*}
If we take $\alpha=-\pi^diz$ and $\beta=\pi^di/z$ in the above equation and simplify then it becomes
\begin{align*}
\sum_{n=1}^\infty n V_\K(n) \Omega_\K\left(\frac{(2\pi)^d n iz}{D_\K}\right)+\frac{H\zeta_\K(-1)}{2\mathcal{C}_\K}= -\frac{1}{z}\left\{\sum_{n=1}^\infty n V_\K(n) \Omega_\K\left(\frac{(2\pi)^d ni}{zD_\K}\right)+\frac{H\zeta_\K(-1)}{2\mathcal{C}_\K}\right\} + \frac{\zeta_\K^2(0)}{\pi i}.
\end{align*}
By employing the definition of $\mathcal{G}_{2,\K}(z)$ as given in \eqref{General Eisenstein} and performing simplifications, we obtain \eqref{quasi1}.
\end{proof}

The transformation formula for logarithmic of Dedekind eta function over fields is proved next.
\begin{proof}[Theorem \textup{\ref{Generalized Dedekind eta transformation}}][]
The argument presented here is similar to that given for the proof of Theorem \ref{gen formula}, and thus our explanation will be terse. We invoke Lemma \ref{inverse mellin transform of omega} so that
\begin{align}\label{L11}
\sum_{n=1}^\infty\frac{V_\K(n)\Omega_\K(ny)}{n}
&=\frac{1}{2\pi i}\int_{(c)}\zeta_\K(1+s)\zeta_\K(s)\frac{\Gamma(s)^{r_1+r_2}}{\Gamma(1-s)^{r_2}}\cos\left(\frac{\pi s}{2}\right)^{r_1-1}y^{-s}ds,
\end{align} 
Combining \eqref{equv fe1} with \eqref{L11}, we see that
\begin{align}\label{L(y) identity1}
\sum_{n=1}^\infty\frac{V_\K(n)\Omega_\K(ny)}{n}=\sqrt{D_\K}2^{r_2-d}\pi^{r_1+r_2-d}\frac{1}{2\pi i}\int_{(c)}\frac{\zeta_\K(1+s)\zeta_\K(1-s)}{\cos\left(\frac{\pi s}{2}\right)}\left(\frac{yD_\K}{(2\pi)^d}\right)^{-s}ds.
\end{align}
We now shift the line of integration from $c>1$ to $-3<\mu=\mathrm{Re}(s)<-1$ and consider the positively oriented contour formed by the line segments $[\mu-iT,c-iT],\ [c-iT,c+iT],\ [c+iT,\mu+iT],\ [\mu+iT,\mu-iT]$. Note that the integrand has double pole at $s=0$ and simple poles at $s=1$ and $s=-1$. Furthermore, using Lemma \ref{Convexity bound} and \eqref{strivert}, it can be shown that the integrals along the horizontal segments vanish as $T\to\infty$. By applying Cauchy's residue theorem along with \eqref{L(y) identity1}, we obtain
\begin{align}\label{gen cauchy1}
\sum_{n=1}^\infty\frac{V_\K(n)\Omega_\K(ny)}{n}=\sqrt{D_\K}2^{r_2-d}\pi^{r_1+r_2-d}\left(R_0+R_{1}+R_{-1}+I(y)\right),
\end{align}
where 
\begin{align}\label{defnI}
I(y):=\frac{1}{2\pi i}\int_{(\mu)}\frac{\zeta_\K(1+s)\zeta_\K(1-s)}{\cos\left(\frac{\pi s}{2}\right)}\left(\frac{yD_\K}{(2\pi)^d}\right)^{-s}ds.
\end{align}
We can evaluate the residues in \eqref{gen cauchy1} as
\begin{align}\label{residues gen1}
R_0&=H^2 \log\left( \frac{yD_\K}{(2\pi)^d} \right)\nonumber\\
R_{1}&=-\frac{2(2\pi)^d}{\pi yD_\K}\zeta_\K(2)\zeta_\K(0).\nonumber\\
R_{-1}&=\frac{2yD_\K}{\pi (2\pi)^d}\zeta_\K(0)\zeta_\K(2).
\end{align}
Employing the change of variable $s=-w$ in \eqref{defnI}, we have, for $1<c=\mathrm{Re}(s)<3$,
\begin{align}\label{I(y)}
I(y)&=\frac{1}{2\pi i}\int_{(c)}\frac{\zeta_\K(1-w)\zeta_\K(1+w)}{\cos\left(\frac{\pi w}{2}\right)}\left(\frac{yD_\K}{(2\pi)^d}\right)^{w}ds\nonumber\\
&=\left(\sqrt{D_\K}2^{r_2-d}\pi^{r_1+r_2-d}\right)^{-1}\sum_{n=1}^\infty\frac{V_\K(n)}{n}\Omega_\K\left(\frac{n(2\pi)^{2d}}{yD_\K^2}\right),
\end{align}
where we used \eqref{L(y) identity1}. Now equations \eqref{gen cauchy1} \eqref{residues gen1} and \eqref{I(y)} together yield
\begin{align}
\sum_{n=1}^\infty\frac{V_\K(n)\Omega_\K(ny)}{n}&=\sqrt{D_\K}2^{r_2-d}\pi^{r_1+r_2-d}\left\{H^2 \log\left( \frac{yD_\K}{(2\pi)^d} \right)-\frac{2(2\pi)^d}{\pi yD_\K}\zeta_\K(2)\zeta_\K(0)+ \frac{2yD_\K}{\pi (2\pi)^d}\zeta_\K(0)\zeta_\K(2)\right\} \nonumber\\
&\quad +\sum_{n=1}^\infty\frac{V_\K(n)}{n}\Omega_\K\left(\frac{n(2\pi)^{2d}}{yD_\K^2}\right).\nonumber
\end{align}
Thus we have
\begin{align}
\sum_{n=1}^\infty\frac{V_\K(n)\Omega_\K(ny)}{n} - \sum_{n=1}^\infty\frac{V_\K(n)}{n}\Omega_\K\left(\frac{n(2\pi)^{2d}}{yD_\K^2}\right)& = \frac{\sqrt{D_{\K}}H^2}{2^{r_1}(2\pi)^{r_2}} \log\left( \frac{yD_\K}{(2\pi)^d} \right)\nonumber\\
&\quad + \frac{2}{\pi} \frac{\sqrt{D_{\K}} \zeta_{\K}(0) \zeta_{\K}(2)}{2^{r_1}(2\pi)^{r_2}} \left( \frac{yD_\K}{(2\pi)^d} - \frac{(2\pi)^d}{yD_\K} \right).\nonumber
\end{align}
Letting $y=\frac{2^d\alpha}{D_\K}$ with $\alpha\beta=\pi^{2d}$ in the above equation we conclude our result.
\end{proof}

\vspace{.4cm}
\noindent
\textbf{Acknowledgements.}
The first author is an INSPIRE Faculty at IISER Kolkata funded by the grant DST/INSPIRE/04/2021/002753 of DST, Govt. of India. Partial work of this project was carried out when the second author was a post-doctoral fellow at the  Institute of Mathematics, Academia Scinica, Taiwan and wants to thank the institute for the academic support. The third author's research was supported by the Fulbright-Nehru Postdoctoral Fellowship Grant 2846/FNPDR/2022. All three authors sincerely thank the respective funding agencies for their support.

\end{document}